\documentclass{zzz} 
\usepackage{graphics,epsfig,color}

\numberwithin{equation}{section}
\numberwithin{theorem}{section}
\numberwithin{proposition}{section}
\numberwithin{lemma}{section}
\numberwithin{corollary}{section}
\numberwithin{definition}{section}
\numberwithin{example}{section}
\numberwithin{remark}{section}
\numberwithin{note}{section}

\newcommand{\N}{{\mathbf N}}

\newcommand{\R}{{\mathbf R}}

\newcommand{\Z}{{\mathbf Z}}
\newcommand{\C}{{\mathbf C}}

\newcommand{\No}{\N_0}

\newcommand{\ac}{\overline{a}}
\newcommand{\acc}{2\Re a}
\newcommand{\bc}{\overline{b}}
\newcommand{\bcc}{2\Re b}
\newcommand{\cc}{\overline{c}}
\newcommand{\ccc}{2\Re c}

\newcommand{\Hyper}[5]{{}_{#1}F_{#2} \left(\!\!\begin{array}{c} {#3} \\ {#4} \end{array}\!;{#5}\right)}

\def\cprime{$'$}

\newtheorem{thm}[lemma]{Theorem}
\newtheorem{cor}[lemma]{Corollary}

\makeatletter
\def\eqnarray{\stepcounter{equation}\let\@currentlabel=\theequation
\global\@eqnswtrue
\tabskip\@centering\let\\=\@eqncr
$$\halign to \displaywidth\bgroup\hfil\global\@eqcnt\z@
  $\displaystyle\tabskip\z@{##}$&\global\@eqcnt\@ne
  \hfil$\displaystyle{{}##{}}$\hfil
  &\global\@eqcnt\tw@ $\displaystyle{##}$\hfil
  \tabskip\@centering&\llap{##}\tabskip\z@\cr}

\def\endeqnarray{\@@eqncr\egroup
      \global\advance\c@equation\m@ne$$\global\@ignoretrue}

\def\@yeqncr{\@ifnextchar [{\@xeqncr}{\@xeqncr[5pt]}}
\makeatother

\begin{document}

\renewcommand{\PaperNumber}{***}

\FirstPageHeading

\ShortArticleName{
Generalizations of
generating functions
for
higher continuous GHOPs}

\ArticleName{Generalizations of generating functions for higher continuous hypergeometric orthogonal polynomials in the Askey scheme}

\Author{Michael A.~Baeder,$^\dag\!\!\ $ Howard S.~Cohl,$^\ddag$
and Hans Volkmer\,$^\S$} 
\AuthorNameForHeading{M.~A.~Baeder, H.~S.~Cohl \& H.~Volkmer}

\Address{$^\dag$~Department of Mathematics, Harvey Mudd College, 
340 East Foothill Boulevard, Claremont, CA 91711, USA
} 
\EmailD{mbaeder@hmc.edu} 

\Address{$^\ddag$~Applied and Computational Mathematics Division, 
National Institute of Standards and Technology, 
Gaithersburg, MD 20899-8910, USA
}
\EmailD{howard.cohl@nist.gov} 

\Address{$^\S$~Department of Mathematical Sciences, 
University of Wisconsin-Milwaukee, P.O. Box 413, 
Milwaukee, WI  53201, USA
}
\EmailD{volkmer@uwm.edu} 


\ArticleDates{Received XX September 2012 in final form ????; Published online ????}

\Abstract{
We use connection relations and series rearrangement to generalize 
generating functions for several higher continuous orthogonal polynomials in 
the Askey scheme, namely the Wilson, continuous dual Hahn, continuous Hahn, 
and Meixner-Pollaczek polynomials. We also determine corresponding definite 
integrals using the orthogonality relations for these polynomials.
} 

\Classification{33C45, 33C20, 34L10, 30D10}


\section{Introduction}
In Cohl (2013) \cite{CohlGenGegen} (see (2.1) therein), we developed 
a series rearrangement 
technique which produced a generalization of the generating function for 
Gegenbauer polynomials.  We have since demonstrated that this technique is valid 
for a larger class of orthogonal polynomials. For instance, in Cohl (2013) 
\cite{Cohl12pow}, we applied this same technique to Jacobi polynomials and
in Cohl, MacKenzie \& Volkmer (2013) \cite{CohlMacKVolk}, we extended this 
technique to many generating functions for Jacobi, Gegenbauer, Laguerre, 
and Wilson polynomials.

The series rearrangement technique starts by combining a connection relation 
with a generating function. This results in a series with multiple sums. The 
order of summations are then rearranged to produce a generalized generating 
function.
This technique is especially productive when using connection relations 
with one free parameter. In this case, the connection relation is usually 
a product of Pochhammer symbols and the resulting generalized generating 
function has coefficients given in terms of generalized hypergeometric 
functions. 

In this paper, we continue this procedure by generalizing generating functions 
for the remaining hypergeometric orthogonal polynomials in the Askey scheme 
\cite[Chapter 9]{Koekoeketal} with continuous orthogonality relations. 
We have also computed definite integrals corresponding to 
our generalized generating function expansions using continuous
orthogonality relations.
The orthogonal polynomials that we treat in this paper are the Wilson, 
continuous dual Hahn, continuous Hahn, and Meixner-Pollaczek polynomials.  
The generalized generating functions we produce through series 
rearrangement usually arise using connection relations with one free parameter.
While connection relations with one free parameter are preferred for 
their simplicity, relations with more free parameters were considered 
when necessary.  

Hypergeometric orthogonal polynomials with more than one free parameter,
such the Wilson polynomials, have connection relations with more than
one free parameter. These connection relations are in general given by 
single or multiple summation expressions. For the Wilson 
polynomials, the connection relation with four free parameters is
given as a double hypergeometric series. The fact that the four free parameter
connection coefficient for Wilson polynomials is given by a double sum was 
known to Askey and Wilson as far back as 1985 (see \cite[p.~444]{Ismail}).
When our series rearrangement technique is applied to cases with more
than one free parameter, the resulting coefficients of the generalized 
generating function are rarely given in terms of a generalized hypergeometric 
series. The more general problem of generalized generating functions with 
more than one free parameter requires the theory of multiple hypergeometric 
series and is not treated in this paper. However, in certain cases when applying 
the series rearrangement technique to generating functions using 
connection relations with one free parameter, the generating function 
remains unchanged.  In these cases, we have found that the introduction
of a second free parameter can sometimes yield generalized generating functions
whose coefficients are given in terms of generalized hypergeometric 
series (see for instance, Section \ref{ContinuousdualHahnpolynomials} below
and \cite[Theorem 1]{Cohl12pow}). 

An interesting question regarding our generalizations is, ``What is the 
origin of specific hypergeometric orthogonal polynomial generating functions?'' 
There only exist two known non-equivalent generating functions for 
the Wilson polynomials, with the Wilson polynomials being at the top of
the Askey scheme.  Unlike the orthogonal 
polynomials in the Askey scheme
which do arise through a limiting procedure from the Wilson polynomials,
most known generating functions for these polynomials do not arrive by this 
same limiting procedure from the two known non-equivalent generating 
functions for Wilson polynomials.  All of the generating functions treated 
in this paper for the continuous Hahn, continuous dual Hahn and 
Meixner-Pollaczek polynomials, as well as most of those generating functions 
treated in our previous papers, do not arrive by this same limiting procedure 
from the Wilson polynomial generating functions. 
Therefore, the generalized generating functions for non-Wilson polynomials 
we present in this paper are interesting by themselves.

Here, we provide a brief introduction into the symbols and special functions 
used in this paper.  We denote the real and complex numbers 
by $\R$ and $\C$, respectively. Similarly,
the sets $\N={1,2,3,\ldots}$ and $\Z={0,\pm1,\pm2,\ldots}$ denote the natural numbers
and the integers. We also use the notation $\No=\{0,1,2,\ldots\}=\N \cup \{0\}$.
If $a_1,a_2,a_3,\ldots \in \C$, and $i, j \in \Z$ such that $j < i$, then
$\sum_{n=i}^j a_n = 0$, and $\prod_{n=i}^j a_n = 1$.
Let $z \in \C$, $n \in \No$. Define the Pochhammer symbol, or
rising factorial, by

\begin{equation} \label{Pochdef}
(z)_n := (z)(z+1)\cdots(z+n-1)=\prod_{i=1}^n (z+i-1).\nonumber
\end{equation}
When $z \notin -\No$, we may also represent the Pochhammer symbol as

\begin{equation} \label{Pochgamma}
(z)_n = \frac{\Gamma(z+n)}{\Gamma(z)},\nonumber
\end{equation}
where $\Gamma:\C\setminus-\No\to\C$ is the gamma function (see
Olver {\it et al.} (2010) \cite[Chapter 5]{NIST}).
Let $a_1,\ldots,a_p \in \C$, and $b_1,\ldots,b_q
\in \C \setminus -\No$. The generalized hypergeometric function ${}_pF_q$ is defined as
\begin{equation} \label{HyperGeometricDef}
\Hyper{p}{q}{a_1,\ldots,a_p}{b_1,\ldots,b_q}{z}:= \sum_{n=0}^\infty 
\frac{(a_1)_n\cdots(a_p)_n}{(b_1)_n\cdots(b_q)_n}\frac{z^n}{n!}.\nonumber
\end{equation}
If $p\le q$ then ${}_pF_q$ is defined for all $z\in\C$. If $p=q+1$ then ${}_pF_q$ is defined in the unit disk $|z|<1$, and can be continued analytically to $\C\setminus[1,\infty)$.
The generalized hypergeometric function is used in the definitions of
hypergeometric orthogonal polynomials and for the coefficients of our 
generalized generating functions. 
\section{Wilson polynomials}
\label{Wilsonpolynomials}
Koekoek {\it et al.} (2010) \cite[(9.1.1)]{Koekoeketal} define the Wilson polynomials $W_n(x^2;a,b,c,d)$ by
\begin{equation} \nonumber
W_n(x^2;a,b,c,d):=(a+b)_n(a+c)_n(a+d)_n\,\Hyper{4}{3}{-n,n+a+b+c+d-1,a+ix,a-ix}{a+b,a+c,a+d}{1},
\end{equation}
where the parameters $a,b,c,d$ are positive, or complex with positive real parts
occurring in conjugate pairs. It is known that $W_n(x^2;a,b,c,d)$ is a symmetric polynomial in the parameters $a,b,c,d$.
The Wilson polynomials occupy the highest echelon of
the Askey scheme, and using limit relations and special parameter values it is
possible to obtain many other hypergeometric orthogonal polynomials -- see for
example Chapter 9.1 of \cite{Koekoeketal}.
S{\'a}nchez-Ruiz and Dehesa (1999) \cite[equation just below (15)]{SanchezDehesa} 
and others previously (see for instance Askey \& Wilson (1985) 
\cite{AskeyWilson}) have given a connection relation for the Wilson polynomials with one free parameter:
\begin{eqnarray} \label{W:connect1}
&& W_n(x^2;a,b,c,d)=\sum_{k=0}^n \frac{n!}{k!(n-k)!} W_k(x^2;a,b,c,h) \nonumber \\
&& \hspace{5mm} 
\times \frac{(n+a+b+c+d-1)_k(d-h)_{n-k}(k+a+b)_{n-k}(k+a+c)_{n-k}(k+b+c)_{n-k}}{
(k+a+b+c+h-1)_k(2k+a+b+c+h)_{n-k}}.
\end{eqnarray}
This is a special case of a more general identity given by S{\'a}nchez-Ruiz and
Dehesa which gives the connection coefficients for the Wilson polynomials with
three free parameters:
\begin{eqnarray}
W_n(x^2;a,b,c,d) && = \sum_{k=0}^n \binom{n}{k} \frac{(n+a+b+c+d-1)_k(k+a+b)_{n-k}
(k+a+c)_{n-k}(k+a+d)_{n-k}}{(k+a+f+g+h-1)_k} \nonumber \\
&& \times \,\Hyper{5}{4}{k-n,k+n+a+b+c+d-1,k+a+f,k+a+g,k+a+h}{2k+a+f+g+h,k+a+b,k+
a+c,k+a+d}{1} \nonumber \\
&& \times W_k(x^2;a,f,g,h) \label{W:connect2}.
\end{eqnarray}
This identity, combined with limit relations, is useful for
deriving connection coefficients for hypergeometric orthogonal polynomials 
lower down in the Askey scheme.

In the process of generalizing generating functions for these lower hypergeometric
orthogonal polynomials, it may be necessary to rearrange the terms in a double or
triple sum. In order to do so, we show the absolute convergence of that
double sum. Therefore, we develop upper bounds
for the quantities of interest. We rely on several useful bounds for Pochhammer
symbols and factorials given by \cite[(48)-(53)]{CohlMacKVolk}. Let $j \in \N$, $k,n \in
\No$, $\Re u>0$ and $v\in\C$. Then
\begin{eqnarray}
&& |(u)_j| \ge (\Re u)(j-1)!, \label{W:bound1} \\
&& \frac{|(v)_n|}{n!} \le (1+n)^{|v|}, \label{W:bound2} \\
&& |(k+v)_{n-k}| \le (1+n)^{|v|}\frac{n!}{k!}, \hspace{2mm} (k\le n), \label{W:bound3} \\
&&  \left|\frac{(k+v)_n}{(k+u)_n}\right|\le \max\{(\Re u)^{-1},1\} (1+n)^{1+|v|}. \label{W:bound4}
\end{eqnarray}

\begin{thm} \label{W:t1}
Let $\rho\in\C$, $|\rho|<1$, $x\in(0,\infty)$, and $a,b,c,d,h$ complex parameters with positive real parts, non-real parameters occurring in conjugate pairs among $a,b,c,d$ and $a,b,c,h$.
Then
\begin{eqnarray} \nonumber
&& \Hyper{2}{1}{a+ix, c+ix}{a+c}{\rho}\Hyper{2}{1}{b-ix, d-ix}{b+d}{\rho}
=
\sum_{k=0}^\infty \frac{(k+a+b+c+d-1)_k}{(k+a+b+c+h-1)_k(a+c)_k(b+d)_kk!} \\
&& \hspace{15mm} \times
\Hyper{4}{3}{d-h, 2k+a+b+c+d-1, k+a+b, k+b+c}{k+a+b+c+d-1, 2k+a+b+c+h, k+b+d}{\rho}
\rho^kW_k(x^2;a,b,c,h). \label{W:f1}
\end{eqnarray}
\end{thm}
\begin{proof}
A generating function for the Wilson polynomials is given by
\cite[(9.1.13)]{Koekoeketal}, namely:
\begin{equation}\label{W:2GF}
\Hyper{2}{1}{a+ix,c+ix}{a+c}{\rho}\,\Hyper{2}{1}{b-ix,d-ix}{b+d}{\rho}=\sum_{n=0}
^\infty \frac{W_n(x^2;a,b,c,d)\rho^n}{(a+c)_n(b+d)_nn!}.
\end{equation}
Substituting \eqref{W:connect1} into  \eqref{W:2GF} gives the
double sum
\begin{eqnarray} \label{W:subs}
\Hyper{2}{1}{a+ix,c+ix}{a+c}{\rho}\,\Hyper{2}{1}{b-ix,d-ix}{b+d}{\rho}=\sum_{n=0}
^\infty c_n \sum_{k=0}^n a_{n,k} W_k(x^2;a,b,c,h),
\end{eqnarray}
where
\begin{eqnarray*}
 c_n=\frac{\rho^n}{(a+c)_n(b+d)_nn!},
\end{eqnarray*}
and
$a_{n,k}$ are the coefficients satisfying
\begin{eqnarray}\label{W:a}
W_n(x^2;a,b,c,d) = \sum_{k=0}^n a_{n,k}W_k(x^2;a,b,c,h).
\end{eqnarray}
In order to justify reversing the order of summation, we show that
\begin{equation*}
\sum_{n=0}^\infty |c_n| \sum_{k=0}^n |a_{n,k}||W_k(x^2;a,b,c,h)| < \infty.
\end{equation*}

Using \eqref{W:bound1}, we see that
\begin{equation}\label{W:boundc}
|c_n| \le K_1\frac{|\rho|^n(1+n)^2}{(n!)^3},
\end{equation}
where $K_1 = \max\left\{1, (\Re(a+c)\Re(b+d))^{-1}\right\}$.
It follows from \cite[(47) and (60)]{CohlMacKVolk} that
\begin{equation}\label{W:boundaW}
\sum_{k=0}^n |a_{n,k}| |W_k(x^2;a,b,c,h)| \le K_2(n!)^3(1+n)^{\sigma_2},
\end{equation}
where $K_2$ and $\sigma_2$ are positive constants independent of $n$.
Combining \eqref{W:boundc} and \eqref{W:boundaW}, we see that
\begin{equation*}
\sum_{n=0}^\infty |c_n| \sum_{k=0}^n |a_{n,k}||W_k(x^2;a,b,c,h| \le K_1K_2\sum
_{n=0}^\infty |\rho|^n(1+n)^{\sigma_2+2} < \infty
\end{equation*}
since $|\rho| < 1$. Reversing the summation, shifting the $n$ summation index by $k$, and simplifying completes the proof.
\end{proof}

\begin{thm} \label{W:t2}
Let $\rho\in\C$, $|\rho|<1$, $x\in(0,\infty)$, and $a,b,c,d,h$ complex parameters with positive real parts, non-real parameters occurring in conjugate pairs among $a,b,c,d$ and $a,b,c,h$.
Then
\begin{eqnarray}
\nonumber
&& (1-\rho)^{1-a-b-c-d}
\Hyper{4}{3}{\frac{1}{2}(a+b+c+d-1), \frac{1}{2}(a+b+c+d), a+ix, a-ix}{a+b, a+c, a+d}{-\frac{4\rho}{(1-\rho)^2}} \\
\nonumber
&& \hspace{10mm} = \sum_{k=0}^\infty \frac{(k+a+b+c+d-1)_k(a+b+c+d-1)_k}{(k+a+b+c+h-1)_k(a+b)_k(a+c)_k(a+d)_kk!} \\
&& \hspace{20mm} \times
\Hyper{3}{2}{2k+a+b+c+d-1, d-h, k+b+c}{2k+a+b+c+h, a+d+k}{\rho}\rho^kW_k(x^2;a,b,c,h).
\label{W:f2}
\end{eqnarray}
\end{thm}
\begin{proof}
Another generating function for the Wilson polynomials is given by
\cite[(9.1.15)]{Koekoeketal}
\begin{eqnarray} \nonumber
&& (1-\rho)^{1-a-b-c-d}\,\Hyper{4}{3}{\frac{1}{2}(a+b+c+d-1),\frac{1}{2}(a+b+c+d),
a+ix,a-ix}{a+b,a+c,a+d}{-\frac{4\rho}{(1-\rho)^2}} \\
&& \hspace{15mm} = \sum_{n=0}^\infty \frac{(a+b+c+d-1)_n}{(a+b)_n(a+c)_n(a+d)_nn!}W_n(x^2;a,b,c,d)\rho^n.
\label{W:4GF}
\end{eqnarray}
It should be noted that $\rho\mapsto -\frac{4\rho}{(1-\rho)^2}$ maps the unit disk $|\rho|<1$ bijectively onto the cut plane $\C\setminus[1,\infty)$, so the left-hand side of \eqref{W:4GF}
is well-defined and analytic for $|\rho|<1$.

Substituting \eqref{W:connect1} into \eqref{W:4GF} gives the double sum
\begin{eqnarray}
&& (1-\rho)^{1-a-b-c-d}\,\Hyper{4}{3}{\frac{1}{2}(a+b+c+d-1),\frac{1}{2}(a+b+c+d),
a+ix,a-ix}{a+b,a+c,a+d}{-\frac{4\rho}{(1-\rho)^2}} \nonumber \\
&& = \sum_{n=0}^\infty c_n \sum_{k=0}^n a_{n,k}W_k(x^2;a,b,c,h),
\end{eqnarray}
where
\begin{equation*}
c_n = \frac{(a+b+c+d-1)_n\rho^n}{(a+b)_n(a+c)_n(a+d)_nn!},
\end{equation*}
and $a_{n,k}$ are the connection coefficients satisfying \eqref{W:a}. We wish to show
\begin{eqnarray}
\sum_{n=0}^\infty |c_n| \sum_{k=0}^n |a_{n,k}||W_k(x^2;a,b,c,h)| < \infty.
\end{eqnarray}
By \eqref{W:bound2}, we determine that
\begin{eqnarray*}
\left|\frac{(a+b+c+d-1)_n}{n!}\right| \le (1+n)^{|a+b+c+d-1|},
\end{eqnarray*}
and thus, by \eqref{W:bound1},
\begin{eqnarray}\label{W:boundc2}
|c_n| \le K_1\frac{(1+n)^{\sigma_1}}{(n!)^3},
\end{eqnarray}
where $\sigma_1=|a+b+c+d-1|+3$ and
$K_1 = \max\left\{1,(\Re(a+b)\Re(a+c)\Re(a+d))^{-1}\right\}$.

Combining \eqref{W:boundaW} and \eqref{W:boundc2}, we see that
\begin{equation*}
\sum_{n=0}^\infty |c_n| \sum_{k=0}^n |a_{n,k}||W_k(x^2;a,b,c,h)| \le K_1K_2
\sum_{n=0}^\infty (1+n)^{\sigma_1+\sigma_2}|\rho|^n < \infty
\end{equation*}
since $|\rho|<1$. Swapping the sums, shifting the inner index, and simplifying
gives the desired result.
\end{proof}


\section{Continuous dual Hahn polynomials}
\label{ContinuousdualHahnpolynomials}

The continuous dual Hahn polynomials are defined by \cite[(9.3.1)]{Koekoeketal}
\begin{equation*} \label{CDH:def}
S_n(x^2;a,b,c):= (a+b)_n(a+c)_n \, \Hyper{3}{2}{-n,a+ix,a-ix}{a+b,a+c}{1},
\end{equation*}
where $a,b,c >0$, except for possibly a pair of complex conjugates with positive real
parts. It is known that $S_n(x^2;a,b,c)$ is a symmetric polynomial in the parameters $a,b,c$.

In order to generalize generating functions of the continuous dual
Hahn polynomials, it is necessary to derive the connection coefficients for the
continuous dual Hahn polynomials with two free parameters.

\begin{lemma}
Let $x \in (0,\infty)$, and $a,b,c,f,g \in \C$ with positive
real parts and non-real values appearing in conjugate pairs among $a,b,c$ and $a,f,g$. Then
\begin{eqnarray}
S_n(x^2;a,b,c) && = \sum_{k=0}^n {n \choose k}
(k+a+b)_{n-k}(k+a+c)_{n-k}
\nonumber \\
&& \hspace{1.5cm}\times 
\,\Hyper{3}{2}{k-n,k+a+f,k+a+g}{k+a+b,k+a+c}{1} 
S_k(x^2;a,f,g) 
\label{CDH:connect}.
\end{eqnarray}
\end{lemma}
\begin{proof}
Letting $h \mapsto d$ in \eqref{W:connect2} gives
\begin{eqnarray}
W_n(x^2;a,b,c,d) && = \sum_{k=0}^n \binom{n}{k} \frac{(n+a+b+c+d-1)_k(k+a+b)_{n-k}
(k+a+c)_{n-k}(k+a+d)_{n-k}}{(k+a+f+g+d-1)_k} \nonumber \\
&& \times \,\Hyper{4}{3}{k-n,k+n+a+b+c+d-1,k+a+f,k+a+g}{2k+a+f+g+d,k+a+b,k+
a+c}{1} \nonumber \\
&& \times W_k(x^2;a,f,g,d). \label{CDH:intermediate}
\end{eqnarray}

The limit relation between the Wilson and continuous dual Hahn polynomials is
given by \cite[Section 9.3, Limit Relations]{Koekoeketal}
\begin{eqnarray*} \label{CDH:limit}
\lim_{d\to\infty} \frac{W_n(x^2;a,b,c,d)}{(a+d)_n} = S_n(x^2;a,b,c).
\end{eqnarray*}
We apply this limit relation to reduce the Wilson connection coefficients
to those for the continuous dual Hahn polynomials. Dividing both sides of
\eqref{CDH:intermediate} by $(a+d)_n$, multiplying the right-hand side by
$(a+d)_k/(a+d)_k$, and taking the limit as $d \to \infty$ gives the desired result.
\end{proof}

For the other two generating functions, we use a connection relation for the
continuous dual Hahn polynomials with one free parameter. Let $a,b,c,d > 0$
except for possibly a pair of complex conjugates with positive real parts among $a,b,c$ and $a,b,d$.
Then
\begin{eqnarray} \label{CDH:connect2}
S_n(x^2;a,b,c) = \sum_{k=0}^n \binom{n}{k} (k+a+b)_{n-k} (c-d)_{n-k} S_k(x^2;a,
b,d).
\end{eqnarray}
This relation follows by letting $f \mapsto b, g \mapsto d$ in \eqref{CDH:connect},
and applying the Chu-Vandermonde identity (see \cite[(15.4.24)]{NIST}) to the
resulting hypergeometric function.

We also need the following bound on the continuous dual Hahn polynomials.

\begin{lemma}
Let $x>0$ and $a,b,c \in \C$ with positive real parts and non-real
values occurring in conjugate pairs. Then, for 
$n\in\No,$
\begin{equation} \label{CDH:bound}
|S_n(x^2;a,b,c)| \le K(n!)^2(1+n)^{\sigma},
\end{equation}
where $K$ and $\sigma$ are constants independent of n (and $x$).
\end{lemma}
\begin{proof}
The generating function  \cite[(9.3.12)]{Koekoeketal}
\begin{equation} \label{CDH:1GF}
(1-\rho)^{-c+ix}\,\Hyper{2}{1}{a+ix,b+ix}{a+b}{\rho}=\sum_{n=0}^\infty \frac{
S_n(x^2;a,b,c)}{(a+b)_nn!}\rho^n
\end{equation}
leads to the representation
\begin{eqnarray*}
S_n(x^2;a,b,c)=(a+b)_nn!\sum_{k=0}^n \frac{(a+ix)_k(b+ix)_k}{(a+b)_kk!} \frac{(c-ix)_{n-k}}{(n-k)!} .
\end{eqnarray*}
Straightforward estimation using \eqref{W:bound1} and \eqref{W:bound2} 
gives \eqref{CDH:bound}.
\end{proof}

\begin{lemma}\label{CDH:l6}
Let $b,c,d,f\in\C$ with $\Re b>0$, $\Re c>0$.
Then, for all $\rho\in\C$ with $|\rho|<1$,
\begin{equation}\label{CDH:eq}
 (1-\rho)^{d-c}\Hyper{2}{1}{b-f, d}{b}{\rho}=\sum_{m=0}^\infty \frac{\rho^m}{m!} (c)_m\,
 \Hyper{3}{2}{-m,d, f}{b,c}{1} .
\end{equation}
\end{lemma}
\begin{proof}
On the right-hand inside of \eqref{CDH:eq}, we substitute
\[
\Hyper{3}{2}{-m,d, f}{b,c}{1}=\sum_{\ell=0}^m (-1)^\ell \frac{m!}{(m-\ell)!\ell!} \frac{(d)_\ell(f)_\ell}{(b)_\ell(c)_\ell} ,\]
and reverse sums. The reversal of sums is justified provided $|\rho|<\frac12$.
Then we obtain
\[ \sum_{m=0}^\infty \frac{\rho^m}{m!} (c)_m\, \Hyper{3}{2}{-m,d, f}{b,c}{1}
= (1-\rho)^{-c}\Hyper{2}{1}{d, f}{b}{\frac{\rho}{\rho-1}} .\]
Now \cite[15.8.1]{NIST} yields \eqref{CDH:eq} for $|\rho|<\frac12$. Since the left-hand side of \eqref{CDH:eq}
is an analytic function in the unit disk $|\rho|<1$ and the right-hand side is 
its Maclaurin expansion,
we see that \eqref{CDH:eq} is valid for $|\rho|<1$.
\end{proof}

\begin{thm}\label{CDH:t1}
Let $\rho\in\C$ with $|\rho|<1$, $x\in(0,\infty)$ and $a,b,c,d,f>0$ except for possibly pairs of complex conjugates with positive real parts among $a,b,c$ and $a,d,f$.
Then
\begin{equation}\label{CDH:f1}
(1-\rho)^{-d+ix}\Hyper{2}{1}{a+ix,b+ix}{a+b}{\rho}=
\sum_{k=0}^\infty \frac{S_k(x^2;a,d,f)\rho^k}{(a+b)_k k!}\Hyper{2}{1}{b-f, k+a+d}{k+a+b}{\rho}.
\end{equation}
\end{thm}
\begin{proof}
Substituting \eqref{CDH:connect} for $S_n(x^2;a,b,c)$ in the generating function \eqref{CDH:1GF} yields
\begin{equation}\label{CDH:sub}
(1-\rho)^{-c+ix}\Hyper{2}{1}{a+ix,b+ix}{a+b}{\rho} =\sum_{n=0}^\infty c_n \sum_{k=0}^n a_{n,k} b_{n,k} S_k(x^2;a,d,f)  ,
\end{equation}
where
\begin{eqnarray*}
 c_n&=&\frac{\rho^n}{(a+b)_n n!},\\
 a_{n,k}&=&\binom{n}{k}(k+a+b)_{n-k}(k+a+c)_{n-k} ,\\
 b_{n,k}&=&\Hyper{3}{2}{k-n, k+a+d, k+a+f}{k+a+b, k+a+c}{1}.
\end{eqnarray*}
We wish to reverse the order of summation so we show that
\begin{equation}\label{CDH:cond}
\sum_{n=0}^\infty|c_n|\sum_{k=0}^n|a_{n,k}||b_{n,k}||S_k(x^2;a,d,f)|<\infty .
\end{equation}
Using \eqref{W:bound1} we find
\begin{equation}\label{CDH:boundc}
|c_n|\le K_1 \frac{(1+n)|\rho|^n}{(n!)^2} ,
\end{equation}
where $K_1=\max\{1,(\Re(a+b))^{-1}\}$. Using \eqref{W:bound3} we determine
\[ |a_{n,k}|\le \binom{n}{k}(1+n)^{|a+b|+|a+c|}\frac{(n!)^2}{(k!)^2} .\]
Combining this with \eqref{CDH:bound} yields
\begin{equation}\label{CDH:bounda}
|a_{n,k}||S_k(x^2;a,d,f)|\le K_2 \binom{n}{k}(1+n)^{\sigma_2} (n!)^2 ,
\end{equation}
where $\sigma_2=|a+b|+|a+c|+\sigma$, $K_2=K$ with $\sigma$ and $K$ defined as in \eqref{CDH:bound}.

By \eqref{W:bound4}
\begin{eqnarray}
 |b_{n,k}|&\le &\sum_{s=0}^{n-k} \binom{n-k}{s}\left|\frac{(k+a+d)_s(k+a+f)_s}{(k+a+b)_s(k+a+c)_s}\right|\nonumber\\
 &\le & \sum_{s=0}^{n-k}\binom{n-k}{s} K_3(1+s)^{\sigma_3}\le K_3 2^n(1+n)^{\sigma_3},\label{CDH:boundb}
\end{eqnarray}
where $K_3=\max\{(\Re a)^{-2},1\}$, $\sigma_3=2+|a+d|+|a+f|$.

Combining \eqref{CDH:boundc}, \eqref{CDH:bounda}, \eqref{CDH:boundb}, we find
\begin{eqnarray*}
 \sum_{n=0}^\infty|c_n|\sum_{k=0}^n|a_{n,k}||b_{n,k}||S_k|& \le & K_1K_2K_3\sum_{n=0}^\infty (1+n)^{1+\sigma_2+\sigma_3}|\rho|^n 2^n\sum_{k=0}^n  \binom{n}{k} \\
  &= &  K_1K_2K_3\sum_{n=0}^\infty (1+n)^{1+\sigma_2+\sigma_3}|\rho|^n 4^n.
\end{eqnarray*}
Therefore, condition \eqref{CDH:cond} is satisfied provided that $|\rho|<\frac14$.
When we reverse sums in \eqref{CDH:sub}, we obtain
\begin{eqnarray*}
&&(1-\rho)^{-c+ix}\Hyper{2}{1}{a+ix,b+ix}{a+b}{\rho} \\
&&\hspace{1cm}=\sum_{k=0}^\infty \frac{\rho^k}{k!(a+b)_k} S_k(x^2;a,d,f)
\sum_{m=0}^\infty \frac{\rho^m}{m!}(k+a+c)_m\, \Hyper{3}{2}{-m, k+a+d, k+a+f}{k+a+b, k+a+c}{1} .
\end{eqnarray*}
We now use Lemma \ref{CDH:l6} and obtain \eqref{CDH:f1} for $|\rho|<\frac14$.
Using \eqref{W:bound4} one can show that the right-hand side of \eqref{CDH:f1} converges locally uniformly for $|\rho|<1$, so by analytic continuation, \eqref{CDH:f1} holds for all $\rho\in\C$ with $|\rho|<1$.
\end{proof}

\begin{thm} \label{CDH:t2}
Let $\rho\in\C$, $x\in(0,\infty)$, and $a,b,c,d > 0$ except for
possibly a pair of complex conjugates with positive real parts among $a,b,c$ and $a,b,d$. Then
\begin{eqnarray}\label{CDH:f2}
&& e^\rho\,\Hyper{2}{2}{a+ix, a-ix}{a+b, a+c}{-\rho}=
\sum_{k=0}^\infty \frac{\rho^k S_k(x^2;a,b,d)}{(a+b)_k(a+c)_kk!}\,
\Hyper{1}{1}{c-d}{k+a+c}{\rho} .
\end{eqnarray}
\end{thm}
\begin{proof}
Another generating function for the continuous dual Hahn polynomials is
given by \cite[(9.3.15)]{Koekoeketal}
\begin{equation} \label{CDH:4GF}
e^\rho\,\Hyper{2}{2}{a+ix,a-ix}{a+b,a+c}{-\rho}=\sum_{n=0}^\infty \frac{S_n(x^2;a,b,c)}
{(a+b)_n(a+c)_nn!}\rho^n.
\end{equation}
We substitute the term $S_n(x^2;a,b,c)$ using \eqref{CDH:connect2}, which
gives
\begin{equation} \label{CDH:sub2}
e^\rho\,\Hyper{2}{2}{a+ix,a-ix}{a+b,a+c}{-\rho}=\sum_{n=0}^\infty c_n \sum_{k=0}^
n a_{n,k}S_k(x^2;a,b,d),
\end{equation}
where
\begin{eqnarray}
c_n&=&\frac{\rho^n}{(a+b)_n(a+c)_n n!}, \nonumber\\
a_{n,k}&=& \binom{n}{k}(k+a+b)_{n-k}(c-d)_{n-k}.\label{CDH:a}
\end{eqnarray}

We wish to reverse the order of summation, so we show
\begin{eqnarray*}
\sum_{n=0}^\infty |c_n| \sum_{k=0}^n |a_{n,k}||S_k(x^2;a,b,d)|<\infty.
\end{eqnarray*}
Using \eqref{W:bound1}, we determine
\begin{equation} \label{CDH:boundc2}
|c_n| \le K_1\frac{(1+n)^{\sigma_1}|\rho|^n}{(n!)^3},
\end{equation}
where $\sigma_1=2$, $K_1=\max\{1,(\Re a)^{-2}\}$ .
Using \eqref{W:bound2} and \eqref{W:bound3}, we determine
\begin{equation} \label{CDH:bounda2}
|a_{n,k}| \le \frac{(n!)^2}{(k!)^2}(1+n)^{\sigma_2},
\end{equation}
where $\sigma_2 = |c-d|+|a+b|$. Combining this with \eqref{CDH:bound}, we see that
\begin{equation} \label{CDH:boundaS}
\sum_{k=0}^n |a_{n,k}S_k(x^2;a,b,d)| \le \sum_{k=0}^n K(n!)^2(1+n)^{\sigma_2
+\sigma} = K(1+n)^{\sigma_2+\sigma+1}(n!)^2.
\end{equation}
Combining \eqref{CDH:boundc2} and \eqref{CDH:boundaS}, we see that
\begin{equation*}
\sum_{n=0}^\infty |c_n| \sum_{k=0}^n |a_{n,k}S_k(x^2;a,b,d)| \le
K_1K \sum_{n=0}^\infty \frac{(1+n)^{\sigma_1+\sigma_2+\sigma+1}|\rho|^n}{n!}
< \infty
\end{equation*}
for any $\rho\in\C$. Reversing the order of summation in \eqref{CDH:sub2},
shifting the $n$ index by $k$, and simplifying gives the desired result.
\end{proof}

\begin{thm} \label{CDH:t3}
Let $\rho\in\C$ with $|\rho|<1$, $x\in(0,\infty)$, $\gamma\in\C$ and $a,b,c,d > 0$ except for
possibly a pair of complex conjugates with positive real parts among $a,b,c$ and $a,b,d$. Then
\begin{eqnarray}
(1-\rho)^{-\gamma} \,\Hyper{3}{2}{\gamma, a+ix, a-ix}{a+b, a+c}{\frac{\rho}{\rho-1}}
&& =
\sum_{n=0}^\infty \frac{(\gamma)_k\rho^k}{(a+b)_k(a+c)_kk!} \nonumber \\
&& \times \,
\Hyper{2}{1}{-d, \gamma+k}{k+a+c}{\rho} S_k(x^2;a,b,d).\label{CDH:f3}
\end{eqnarray}
\end{thm}
\begin{proof}
Yet another generating function for the continuous
dual Hahn polynomials \cite[(9.3.16)]{Koekoeketal} is
\begin{equation} \label{CDH:5GF}
(1-\rho)^{-\gamma}\,\Hyper{3}{2}{\gamma,a+ix,a-ix}{a+b,a+c}{\frac{\rho}{\rho-1}}=
\sum_{n=0}^\infty \frac{(\gamma)_nS_n(x^2;a,b,c)}{(a+b)_n(a+c)_nn!}\rho^n.
\end{equation}
Substituting \eqref{CDH:connect2} into \eqref{CDH:5GF} yields the double sum
\begin{equation} \label{CDH:sub3}
(1-\rho)^{-\gamma}\,\Hyper{3}{2}{\gamma,a+ix,a-ix}{a+b,a+c}{\frac{\rho}{\rho-1}}
= \sum_{n=0}^\infty c_n \sum_{k=0}^n a_{n,k}S_k(x^2;a,b,d),
\end{equation}
where
\begin{eqnarray*}
c_n = \frac{(\gamma)_n\rho^n}{(a+b)_n(a+c)_nn!},
\end{eqnarray*}
and $a_{n,k}$ is defined by \eqref{CDH:a}. By \eqref{W:bound1}, \eqref{W:bound2}, we obtain
\begin{equation} \label{CDH:boundc3}
|c_n| \le K_1\frac{(1+n)^{\sigma_1}|\rho|^n}{(n!)^2},
\end{equation}
where $\sigma_1=2+|\gamma|$ and $K_1 = \max\left\{1,(\Re a)^{-2}\right\}$.

We recall the bound on $a_{n,k}S_k(x^2;a,b,d)$ given by \eqref{CDH:boundaS}. Using
this, we have shown
\begin{eqnarray*}
\sum_{n=0}^\infty |c_n| \sum_{k=0}^n |a_{n,k}||S_k(x^2;a,b,d)| \le K_1K\sum
_{n=0}^\infty (1+n)^{\sigma_1+\sigma_2+\sigma+2}|\rho|^n<\infty,
\end{eqnarray*}
since $|\rho|<1$. Reversing the summation, shifting the $n$ index by $k$, and
simplifying completes the proof.
\end{proof}


\section{Continuous Hahn polynomials}

The continuous Hahn polynomial with two pairs of conjugate parameter 
are defined by \cite[(9.4.1)]{Koekoeketal}
\begin{eqnarray*}\label{CH:def}
p_n(x;a,b,\ac,\bc):=i^n\frac{(\acc)_n(a+\bc)_n}{n!}\,\Hyper{3}{2}{-n,n+\acc+\bcc-1,a+ix}
{\acc,a+\bc}{1},
\end{eqnarray*}
where $a$ and $b$ are complex numbers with positive real parts. 
The continuous Hahn polynomials are symmetric in $a,b$.
The limit relation between the continuous Hahn and Wilson polynomials is given by
\cite[Section 9.4, Limit Relations]{Koekoeketal}
\begin{equation} \label{CH:limit}
\lim_{t\to\infty} \frac{W_n((x+t)^2;a-it,b-it,\ac+it,\bc+it)}{(-2t)^nn!}=p_n(x;
a,b,\ac,\bc).
\end{equation}

In the following results for continuous Hahn polynomials, we assume the 
restriction $\Im a=\Im b=\Im c.$ 
This is because in the general case, one can not transform 
the ${}_3F_2$ in the connection relation with one free parameter (See (4.3) below) 
to a simple product of gamma functions.
The case $\Im a=\Im b=\Im c$ is the most general complex solution to the
problem of obtaining a generalized generating function for continuous Hahn polynomials
using a connection relation with one free parameter.
It is interesting to note however that the case $\Im a=\Im b=\Im c$ can be reduced to the
case $\Im a=\Im b=\Im c=0,$ because the change of $a,b,c>0$ to
$a+ih,b+ih,c+ih$ only leads to a shift in the $x$-variables of polynomials involved.

\begin{thm}\label{CH:l1}
Let $a,b,c\in\C$ such that $\Re a>0$, $\Re b>0$, $\Re c>0$ and
$\Im a = \Im b = \Im c$.     Then
\begin{eqnarray}\label{CH:specialconnect} 
 p_n(x;a,b,\ac,\bc)=\sum_{k=0}^n a_{n,k} p_k(x;a,c,\ac,\cc),
\end{eqnarray}
where
$a_{n,k}=0$ if $n-k$ is odd and, if $n=k+2p$, $p\in\No$,
\begin{eqnarray*} 
a_{n,k}=
\frac{(-1)^p(n+\acc+\bcc-1)_k(k+\acc)_{n-k}(k+a+\bc)_{n-k}(b-c)_p}
{4^pp!(k+\acc+\ccc-1)_k(k+a+\cc+\frac12)_p(a+\bc+k)_p}. 
\end{eqnarray*}
\end{thm}
\begin{proof}
Let $a \mapsto a-it$, $b \mapsto
b-it$, $c \mapsto \ac+it$, $d \mapsto \bc+it$, $f \mapsto c-it$, $h \mapsto \cc+it$,
and $x \mapsto x+t$ in \eqref{W:connect2}. Divide both sides by $(-2t)^nn!$, multiply
the right-hand side by $(-2t)^kk!/(-2t)^kk!$, and take the limit as $t \to \infty$. This
yields connection coefficients for the continuous Hahn polynomials with one free
parameter
\begin{eqnarray} \nonumber
p_n(x;a,b,\ac,\bc) && =\sum_{k=0}^n \frac{(n+\acc+\bcc-1)_k(k+\acc)_{n-k}(k+a+\bc)
_{n-k}}{(n-k)!(k+\acc+\ccc-1)_k}i^{n-k}p_k(x;a,c,\ac,\cc) \\
&& \hspace{2cm}\times \, \Hyper{3}{2}{k-n,k+n+\acc+\bcc-1,k+a+\cc}{2k+\acc+\ccc,k+a+\bc}{1}.
\label{CH:connect}
\end{eqnarray}

We further reduce this using Whipple's sum \cite[(16.4.7)]{NIST}
\begin{eqnarray} \label{CH:Whipple}
\hspace{-0.5cm}\Hyper{3}{2}{a',b',c'}{\frac{1}{2}(a'+b'+1),2c'}{1}=\frac{\sqrt{\pi}\,\Gamma(c'+\frac{1}
{2})\Gamma(\frac{1}{2}(a'+b'+1))\Gamma(c'+\frac{1}{2}(1-a'-b'))}{\Gamma(\frac{1}{2}(
a'+1))\Gamma(\frac{1}{2}(b'+1))\Gamma(c'+\frac{1}{2}(1-a'))\Gamma(c'+\frac{1}{2}(1-b'
))}.
\end{eqnarray}
To apply 
(\ref{CH:Whipple}) in (\ref{CH:connect}),
we assume that $\Im a=\Im b=\Im c$. 
Then the hypergeometric series may be rewritten as
\begin{eqnarray*}
&&\Hyper{3}{2}{k-n,k+n+\acc+\bcc-1,k+a+\cc}{2(\Re a+\Re c+k),a+\bc+k}{1}\\
&&\hspace{4cm}=\Hyper{3}{2}{k-n,k+n+
2\Re a+2\Re b-1,k+\Re a+\Re c}{2(\Re a+\Re c+k),\Re a+\Re b+k}{1}.
\end{eqnarray*}
Setting $a' = k-n$, $b'=k+n+2\Re a+2\Re b-1$, and $c'=k+\Re a+\Re c$, we apply \eqref{CH:Whipple}
and determine
\begin{eqnarray}
&& \hspace{-05mm}\Hyper{3}{2}{k-n,k+n+\acc+\bcc-1,k+a+\cc}{2(\Re a+\Re c+k),k+a+\bc}{1}=\frac{\sqrt{\pi}\,\Gamma(
k+\Re a+\Re c+\frac{1}{2})}{\Gamma(\frac{1}{2}(k-
n+1))\Gamma(\frac{1}{2}(k+n+2\Re a+2\Re b))} \nonumber\\
&& \hspace{40mm} \times \frac{\Gamma(\Re a+\Re b+k)\Gamma(\Re c-\Re b+1)}{\Gamma(\frac{1}{2}(k+n+2\Re a+2\Re c+1))
\Gamma(\frac{1}{2}(k-n+2+2\Re c-2\Re b))}.
\label{CH:Whipplesumim}
\end{eqnarray}
It follows from \eqref{CH:connect}, \eqref{CH:Whipplesumim} that $a_{n,k}=0$ if $n-k$ is odd. If $n=k+2p$ with $p\in\No$,
then
\begin{eqnarray}\label{CH:Whipple2}
\hspace{-0.4cm}\Hyper{3}{2}{k-n,k+n+\acc+\bcc-1,k+a+\cc}{2(\Re a+\Re c+k),k+a+\bc}{1}=\frac{(b-c)_p(2p)!}{(k+a+\cc+\frac12)_p(a+\bc+k)_p p! 4^p}.
\end{eqnarray}
Substituting (\ref{CH:Whipple2})
into \eqref{CH:connect} with simplification completes the proof.
\end{proof}

We derive a bound for the continuous Hahn polynomials.

\begin{lemma}\label{CH:l2}
Let $a,b\in\C$ with positive real part, and $x\in\R$. Then, for all $n\in\No$,
\begin{eqnarray}\label{CH:bound}
|p_n(x;a,b,\ac,\bc)|\le n!(1+n)^\sigma,
\end{eqnarray}
where $\sigma$ is a constant independent of $n$.
\end{lemma}
\begin{proof}
Wilson (1991) \cite[(2.2)]{Wilson91} showed that
\begin{eqnarray}\label{CH:Wilson}
W_n(x^2;a,b,c,d)=n! \sum_{k=0}^n u_k(ix)u_{n-k}(-ix)\frac{2ix-n+2k}{2ix} ,
\end{eqnarray}
where
\begin{eqnarray*}
u_k(x):=\frac{(a+x)_k(b+x)_k(c+x)_k(d+x)_k}{k!(1+2x)_k} .
\end{eqnarray*}
Applying the limit relation \eqref{CH:limit} to \eqref{CH:Wilson}, we find
\begin{eqnarray}\label{CH:Wilson2}
p_n(x;a,b,\ac,\bc)=i^n \sum_{k=0}^n (-1)^k \frac{(a+ix)_k(b+ix)_k(\ac-ix)_{n-k}(\bc-ix)_{n-k}}{k!(n-k)!} .
\end{eqnarray}
Using \eqref{W:bound2}, we obtain
\begin{eqnarray*}
&& |p_n(x;a,b,\ac,\bc)|\le \sum_{k=0}^n (1+k)^\sigma k! (1+n-k)^\sigma (n-k)!\\
&& \hspace{2.7cm} \le (1+n)^{2\sigma} \sum_{k=0}^n k!(n-k)!\le (1+n)^{2\sigma+1} n!,
\end{eqnarray*}
where $\sigma=|a+ix|+|b+ix|$.
\end{proof}

\begin{thm} \label{CH:t1}
Let $\rho\in\C$, $x \in \R$, and $a,b,c \in \C$ such that $\Re a>0$, $\Re b>0$, $\Re c>0$ and
$\Im a = \Im b = \Im c$. Then
\begin{eqnarray}
&& \nonumber \Hyper{1}{1}{a+ix}{\acc}{-i\rho}\Hyper{1}{1}{\bc-ix}{\bcc}{i\rho}
= \sum_{k=0}^\infty \frac{(k+\acc+\bcc-1)_k}{(\acc)_k(\bcc)_k(k+\acc+\ccc-1)_k} \rho^k  p_k(x,a,c,\ac,\cc)\\ \nonumber
&& \hspace{20mm} \times \,
\Hyper{4}{5}{\frac{1}{2}(\Re a+\Re b+k),\frac{1}{2}(\Re a+\Re b+k+1),
\Re a+\Re b+k-\frac12
,\Re b-\Re c}
{\Re a+\Re b+\frac{k-1}{2},\Re a+\Re b+\frac{k}2,\Re b+\frac{k}2,\Re b+\frac{k+1}{2},
\Re a+\Re c+k+\frac12}{\frac{-\rho^2}{4}}.
\nonumber\label{CH:f1}
\end{eqnarray}
\end{thm}
\begin{proof}
A generating function for the continuous Hahn polynomials is given
by \cite[(9.4.11)]{Koekoeketal}
\begin{eqnarray} \label{CH:1GF}
\Hyper{1}{1}{a+ix}{\acc}{-i\rho}\Hyper{1}{1}{\bc-ix}{\bcc}{i\rho}=\sum_{n=0}^\infty
\frac{p_n(x;a,b,\ac,\bc)}{(\acc)_n(\bcc)_n}\rho^n.
\end{eqnarray}
Substituting \eqref{CH:specialconnect} in \eqref{CH:1GF} gives the double sum
\begin{eqnarray*}
\Hyper{1}{1}{a+ix}{\acc}{-i\rho}\Hyper{1}{1}{\bc-ix}{\bcc}{i\rho}=\sum_{n=0}^\infty
c_n\sum_{k=0}^n a_{n,k}p_k(x;a,c,\ac,\cc),
\end{eqnarray*}
where
\begin{eqnarray*}
c_n = \frac{\rho^n}{(\acc)_n(\bcc)_n},
\end{eqnarray*}
and the $a_{n,k}$ are the coefficients satisfying \eqref{CH:specialconnect}. 

We wish to reverse the order of summation, so we need to show that
\begin{eqnarray*}
\sum_{n=0}^\infty|c_n|\sum_{k=0}^n|a_{n,k}||p_k(x;a,c,\ac,\cc)|<\infty.
\end{eqnarray*}
Using \eqref{W:bound1}, we determine
\begin{eqnarray*}
|c_n| \le K_1\frac{(1+n)^2|\rho|^n
}{(n!)^2},
\end{eqnarray*}
where $K_1= \max\{1,(4\Re a\Re b)^{-1}\}$.
Using the bounds \eqref{W:bound1}, \eqref{W:bound2}, \eqref{W:bound3}, \eqref{W:bound4}, we estimate 
\begin{eqnarray}\label{CH:bounda}
|a_{n,k}|\le K_2(1+n)^{\sigma_2} \frac{n!k!4^k}{(2k)!} .
\end{eqnarray}
where we used that $\binom{2m}{m}\le 4^m$ with $m=p+k$. By these estimates and \eqref{CH:bound},
\begin{eqnarray*}
\sum_{n=0}^\infty |c_n|\sum_{k=0}^n |a_{n,k}p_k|&\le & K_1K_2\sum_{n=0}^\infty (1+n)^{\sigma+\sigma_2+2}\frac{|\rho|^n}{n!}\sum_{k=0}^n \frac{4^k (k!)^2}{(2k)!}\\
&\le & K_3\sum_{n=0}^\infty (1+n)^{\sigma+\sigma_2+7/2}\frac{|\rho|^n}{n!}<\infty
\end{eqnarray*}
for any $\rho\in\C$, where we used that the sequence $4^{-k}\binom{2k}{k}\sqrt{k}$ converges to $\pi^{-1/2}$.

Therefore, we are justified to reverse summation in the double sum, and we obtain
\[ \sum_{n=0}^\infty c_n\sum_{k=0}^n a_{n,k}p_k(x;a,c,\ac,\cc)=\sum_{k=0}^\infty p_k(x;a,c,\ac,\cc)\sum_{m=0}^\infty c_{m+k}a_{m+k,k} .\]
Since $a_{m+k,k}=0$ for odd $m$, we may set $m=2p$. Then using $a_{2p+k,k}$ as given by Theorem \ref{CH:l1},
we obtain the desired result after some simplification.
\end{proof}

\begin{thm} \label{CH:t2}
Let $\rho\in\C$, $|\rho|<1$, $x \in \R,$  and $a,b,c \in \C$ such that $\Re a>0$, $\Re b>0$, $\Re c>0$ and
$\Im a = \Im b = \Im c$. Then
\begin{eqnarray*} \nonumber
&& (1-\rho)^{1-\acc-\bcc}
\,\Hyper{3}{2}{\Re a+\Re b-\frac12
,\Re a+\Re b
,a+ix}{\acc,a+\bc}{-\frac{4\rho}{(1-\rho)^2}}\\
&=&\sum_{k=0}^\infty \frac{(\acc+\bcc-1)_{2k}}{(\acc)_k(\Re a+\Re b)_k(\acc+\ccc+k-1)_k}
  \,\Hyper{2}{1}{\Re a+\Re b+k-\frac12,\Re b-\Re c}{\Re a+\Re c+k+\frac{1}{2}}{\rho^2}
 \nonumber \\
&& \hspace{15mm} \times \left(-i\rho\right)^kp_k(x;a,c,\ac,\cc).\label{CH:f2}
\end{eqnarray*}
\end{thm}

\begin{proof}
Another generating function for the continuous Hahn polynomials
is given by \cite[(9.4.13)]{Koekoeketal}
\begin{eqnarray} \nonumber
&& (1-\rho)^{1-\acc-\bcc}\Hyper{3}{2}{\Re a+\Re b-\frac{1}{2},\Re a+\Re b,
a+ix}{\acc,a+\bc}{-\frac{4\rho}{(1-\rho)^2}} \\
&& \hspace{30mm} = \sum_{n=0}^\infty \frac{(\acc+\bcc-1)_n}{(\acc)_n(a+\bc)_ni^n}
p_n(x;a,b,\ac,\bc)\rho^n. \label{CH:3GF}
\end{eqnarray}
Note that the left-hand side of \eqref{CH:3GF} is an analytic function of $\rho$ in the unit disk $|\rho|<1$, and \eqref{CH:3GF} is valid for $|\rho|<1$.
We substitute $p_n(x;a,b,\ac,\bc)$ in the generating function \eqref{CH:3GF} using \eqref{CH:specialconnect}.
This gives the double sum
\begin{eqnarray*}
 (1-\rho)^{1-\acc-\bcc}\Hyper{3}{2}{\Re a+\Re b-\frac{1}{2},\Re a+\Re b,
a+ix}{\acc,a+\bc}{-\frac{4\rho}{(1-\rho)^2}} = \sum_{n=0}^\infty c_n \sum_{k=0}^n
a_{n,k}p_k(x;a,c,\ac,\cc),
\end{eqnarray*}
where
\begin{eqnarray*}
c_n = \frac{(\acc+\bcc-1)_n}{(\acc)_n(a+\bc)_ni^n}\rho^n,
\end{eqnarray*}
and $a_{n,k}$ are the coefficients from \eqref{CH:specialconnect}.
Using \eqref{CH:bound} and \eqref{CH:bounda} we show that
\begin{eqnarray*}
\sum_{n=0}^\infty|c_n|\sum_{k=0}^n|a_{n,k}||p_k(x;a,c,\ac,\cc)| \le  K_4\sum_{n=0}^\infty (1+n)^{\sigma+\sigma_3+7/2}|\rho|^n<\infty
\end{eqnarray*}
provided that $|\rho|<1$.
Using (\ref{CH:Whipplesumim}) and manipulations as in the proof of Theorem \ref{CH:t1} gives the desired result
\eqref{CH:f2}.
\end{proof}

\section{Meixner-Pollaczek polynomials}

The Meixner-Pollaczek polynomials are defined as \cite[(9.7.1)]{Koekoeketal}
\begin{eqnarray} \label{MP:def}
P_n^{(\lambda)}(x;\phi):=\frac{(2\lambda)_n}{n!}e^{in\phi}\Hyper{2}{1}{-n,\lambda
+ix}{2\lambda}{1-e^{-2i\phi}}.
\end{eqnarray}
We obtain a connection relation for the Meixner-Pollaczek polynomials with one free
parameter.

\begin{lemma}\label{MP:l1}
Let $\lambda>0, \phi, \psi \in (0,\pi)$. Then
\begin{eqnarray} \label{MP:connect}
P_n^{(\lambda)}(x;\phi)=\frac{1}{\sin^n\psi}\sum_{k=0}^n\frac{(2\lambda+k)_{n-k}}
{(n-k)!}\sin^k\phi\sin^{n-k}(\psi-\phi)P_k^{(\lambda)}(x;\psi).
\end{eqnarray}
\end{lemma}
\begin{proof}
The Meixner-Pollaczek polynomials are obtained from the continuous dual
Hahn polynomials via the limit relation \cite[Section 9.7, Limit Relations]{Koekoeketal}
\begin{eqnarray*} \label{MP:limit}
\lim_{t\to\infty} \frac{S_n((x-t)^2;\lambda+it,\lambda-it,t\cot\phi)}{\left(\frac{
t}{\sin\phi}\right)_nn!}=P_n^{(\lambda)}(x;\phi).
\end{eqnarray*}
Letting $x \mapsto x-t, a \mapsto \lambda+it$, $b \mapsto \lambda-it$, $c \mapsto \phi$, and
$d \mapsto \psi$ in \eqref{CDH:connect2} gives
\begin{eqnarray*}
S_n((x-t)^2;\lambda+it,\lambda-it,t\cot\phi) && =\sum_{k=0}^n\binom{n}{k}(k+2\lambda
)_{n-k}(t(\cot\phi-\cot\psi))_{n-k} \\
&& \times S_k((x-t)^2;\lambda+it,\lambda-it,t\cot\psi).
\end{eqnarray*}

Dividing both sides of this equation by \small$\left(\frac{t}{\sin\phi}\right)_n$
\normalsize and multiplying the right-hand side by \small$\frac{\left(\frac{t}
{\sin\psi}\right)_k}{\left(\frac{t}{\sin\psi}\right)_k}$ \normalsize and taking the limit as
$t \to \infty$ gives the desired result.
\end{proof}

We now derive an upper bound for the Meixner-Pollaczek polynomials.

\begin{lemma}\label{MP:l2}
Let $\lambda>0$, $\phi\in (0,\pi)$ and $x\in\R$. Then
\begin{eqnarray} \label{MP:bound}
|P_n^{(\lambda)}(x;\phi)|\le (1+n)^\sigma,
\end{eqnarray}
where $\sigma$ is a constant independent of $n$.
\end{lemma}
\begin{proof}
The generating
function \cite[(9.7.11)]{Koekoeketal}
\begin{eqnarray} \label{MP:1GF}
(1-e^{i\phi}\rho)^{-\lambda+ix}(1-e^{-i\phi}\rho)^{-\lambda-ix}=\sum_{n=0}^\infty
P_n^{(\lambda)}(x;\phi)\rho^n
\end{eqnarray}
leads to the representation
\begin{eqnarray*}
P_n^{(\lambda)}(x;\phi)=\sum_{k=0}^n \frac{(\lambda-ix)_k}{k!}\frac{(\lambda+ix)_{n-k}}{(n-k)!} e^{i\phi(2k-n)} .
\end{eqnarray*}
Straightforward estimation using \eqref{W:bound2} gives \eqref{MP:bound} with $\sigma=2|\lambda+ix|+1$.
\end{proof}

\begin{thm}\label{MP:t1}
Let $\lambda>0$, $\psi,\phi \in (0,\pi)$, $x\in\R$, and $\rho\in\C$ such that
\begin{eqnarray}\label{MP:cond}
|\rho|(\sin\phi+|\sin(\psi-\phi)|)< \sin \psi .
\end{eqnarray}
Then
\begin{eqnarray}\label{MP:f1}
(1-e^{i\phi}\rho)^{-\lambda+ix}(1-e^{-i\phi}\rho)^{-\lambda-ix}=\left(1-\rho\frac{\sin(\psi-\phi)}{\sin\psi}\right)^{-2\lambda}
\sum_{k=0}^\infty P_k^{(\lambda)}(x;\psi)\tilde \rho^k ,
\end{eqnarray}
where
\begin{eqnarray*}
\tilde\rho= \frac{\rho\sin\phi}{\sin\psi-\rho\sin(\psi-\phi)} .
\end{eqnarray*}
\end{thm}
\begin{proof}
We apply \eqref{MP:connect} to the generating function \eqref{MP:1GF}. 
This yields a double sum
\begin{eqnarray*}
(1-e^{i\phi}\rho)^{-\lambda+ix}(1-e^{-i\phi}\rho)^{-\lambda-ix}=\sum_{n=0}^\infty \rho^n\sum_{k=0}^na_{n,k}P_k^{(\lambda)}(x;\psi),
\end{eqnarray*}
where  $a_{n,k}$ are the coefficients satisfying
\begin{eqnarray}\label{MP:connect2}
P_n^{(\lambda)}(x;\phi)=\sum_{k=0}^n a_{n,k}P_k^{(\lambda)}(x;\psi)
\end{eqnarray}
given explicitly in \eqref{MP:connect}.
In order to justify reversal of summation, we show
\begin{eqnarray*}
\sum_{n=0}^\infty |\rho|^n\sum_{k=0}^n|a_{n,k}|P_k^{(\lambda)}(x;\psi)| < \infty.
\end{eqnarray*}
Using \eqref{W:bound3} we determine
\begin{eqnarray}\label{MP:bounda}
|a_{n,k}| \le \frac{(1+n)^{2\lambda}}{\sin^n\psi}\sum_{k=0}^n \left(n\atop k\right)\sin^k\phi|\sin(\psi-\phi)|^{n-k}=\frac{(1+n)^{2\lambda}}{\sin^n\psi}\left(\sin\phi+|\sin(\psi-\phi)|\right)^n .
\end{eqnarray}
Combining \eqref{MP:bounda} with \eqref{MP:bound}, we see that
\begin{eqnarray*}
\sum_{n=0}^\infty |\rho|^n \sum_{k=0}^n |a_{n,k}||P_k^{(\lambda)}(x;\psi)| \le
\sum_{n=0}^\infty (1+n)^{2\lambda+\sigma+1}\frac{|\rho|^n}{\sin^n\psi}\left(\sin\phi+|\sin(\psi-\phi)|\right)^n <\infty
\end{eqnarray*}
by assumption on $\rho$.
Reversing the summation and simplifying gives the desired result.
\end{proof}

In view of \eqref{MP:1GF}, \eqref{MP:f1} is equivalent to the identity
\begin{eqnarray}\label{MP:f1a}
(1-e^{i\phi}\rho)^{-\lambda+ix}(1-e^{-i\phi}\rho)^{-\lambda-ix}=\left(1-\rho\frac{\sin(\psi-\phi)}{\sin\psi}\right)^{-2\lambda}(1-e^{i\psi}\tilde\rho)^{-\lambda+ix}(1-e^{-i\psi}\tilde\rho)^{-\lambda-ix}.
\end{eqnarray}
Actually, \eqref{MP:f1a} can be verified by elementary but not completely trivial calculations.
Nevertheless, \eqref{MP:f1} is a useful formula. Since it can be derived independently of the connection formula \eqref{MP:connect} it may be used to give a second proof
of Lemma \ref{MP:l1}.

\begin{thm} \label{MP:t2}
Let $\lambda>0$, $\rho\in\C$ and $\psi,\phi \in (0,\pi)$. Then
\begin{eqnarray} \label{MP:f2}
e^\rho \Hyper{1}{1}{\lambda+ix}{2\lambda}{(e^{-2i\phi}-1)\rho}
=\exp{\left(\frac{\rho e^{-i\phi}\sin(\psi-\phi)}{\sin\psi}\right)} \sum_{k=0}^\infty\left(\frac{\sin^k\phi}{\sin^k\psi}\right)\frac{P_k^{(\lambda)}(x;\psi)}{(2\lambda)_ke^{ik\phi}}\rho^k.
\end{eqnarray}
\end{thm}
\begin{proof}
A generating function for the Meixner-Pollaczek polynomials is given
by \cite[(9.7.12)]{Koekoeketal}
\begin{eqnarray} \label{MP:2GF}
e^\rho\Hyper{1}{1}{\lambda+ix}{2\lambda}{(e^{-2i\phi}-1)\rho}=\sum_{n=0}^\infty
\frac{P_n^{(\lambda)}(x;\phi)}{(2\lambda)_ne^{in\phi}}\rho^n.
\end{eqnarray}
To this generating function, we apply \eqref{MP:connect}. This yields a double sum
\begin{eqnarray}\label{MP:dsum1}
e^\rho\Hyper{1}{1}{\lambda+ix}{2\lambda}{(e^{-2i\phi}-1)\rho}=\sum_{n=0}^\infty
c_n\sum_{k=0}^na_{n,k}P_k^{(\lambda)}(x;\psi),
\end{eqnarray}
where
\begin{eqnarray*}
c_n = \frac{\rho^n}{(2\lambda)_ne^{in\phi}}
\end{eqnarray*}
and $a_{n,k}$ are the coefficients satisfying \eqref{MP:connect2}.

We bound $|c_n|$ using \eqref{W:bound1}
\begin{eqnarray} \label{MP:bound1}
|c_n|  \le K_1\frac{(1+n)|\rho|^n}{n!},
\end{eqnarray}
where $K_1= \max\{1,(2\lambda)^{-1}\}$.
We use \eqref{MP:bound}, \eqref{MP:bounda} and \eqref{MP:bound1} to show
\begin{eqnarray*}
\sum_{n=0}^\infty |c_n| \sum_{k=0}^n |a_{n,k}||P_k^{(\lambda)}(x;\psi)| \le
K_1\sum_{n=0}^\infty \frac{(1+n)^{2\lambda+\sigma+1}}{n!} \left(\frac{2|\rho|}{\sin\psi}\right)^n<\infty
\end{eqnarray*}
which justifies reversal of the sum in \eqref{MP:dsum1}.
Reversing the summation and simplifying gives the desired result.
\end{proof}

\begin{thm} \label{MP:t3}
Let $\lambda > 0$, $\gamma, \rho\in\C$ and $\psi, \phi \in (0,\pi)$ be such that \eqref{MP:cond} holds. Then
\begin{eqnarray}
(1-\rho)^{-\gamma}
\Hyper{2}{1}{\gamma,\lambda+ix}{2\lambda}{\frac{(1-e^{-2i\phi})\rho}{\rho-1}}
&& = \left(1-\frac{\rho\sin(\psi-\phi)}{e^{i\phi}\sin\psi}\right)^{-\gamma} \label{MP:f3}\\
&& \times \sum_{k=0}^\infty\left(\frac{\sin^k\phi}{\sin^k\psi}\right)\left(1-\frac{\rho\sin(\psi-\phi)}{e^{i\phi}\sin\psi}\right)^{-k}\frac{(\gamma)_kP_k^{(\lambda)}(x;\psi)}{(2\lambda)_ke^{ik\phi}}\rho^k
\nonumber
\end{eqnarray}
\end{thm}
\begin{proof}
A generating function for the Meixner-Pollaczek polynomials is given
by \cite[(9.7.13)]{Koekoeketal}
\begin{eqnarray} \label{MP:3GF}
(1-\rho)^{-\gamma}\Hyper{2}{1}{\gamma,\lambda+ix}{2\lambda}{\frac{(1-e^{-2i\phi})
\rho}{\rho-1}}=\sum_{n=0}^\infty \frac{(\gamma)_n}{(2\lambda)_n}\frac{P_n^{(\lambda)}
(x;\phi)}{e^{in\phi}}\rho^n.
\end{eqnarray}
We substitute \eqref{MP:connect} for $P_n^{(\lambda)}(x;\phi)$ in this generating
function. This yields the double sum
\begin{eqnarray*}
(1-\rho)^{-\gamma}\Hyper{2}{1}{\gamma,\lambda+ix}{2\lambda}{\frac{(1-e^{-2i\phi})
\rho}{\rho-1}}=\sum_{n=0}^\infty c_n \sum_{k=0}^n a_{n,k}P_n^{(\lambda)}(x;\psi),
\end{eqnarray*}
where
\begin{eqnarray*}
c_n = \frac{(\gamma)_n\rho^n}{(2\lambda)_ne^{in\phi}},
\end{eqnarray*}
and $a_{n,k}$ are the coefficients satisfying \eqref{MP:connect2}.
In order to reverse the order of summation, we show that
\begin{eqnarray*}
\sum_{n=0}^\infty |c_n| \sum_{k=0}^n |a_{n,k}||P_k^{(\lambda)}(x;\phi)| < \infty.
\end{eqnarray*}
To bound $|c_n|$, we apply \eqref{W:bound1} and \eqref{W:bound2}
\begin{eqnarray} \label{MP:bound3}
|c_n| \le K_1(1+n)^{\sigma_1}|\rho|^n,
\end{eqnarray}
where $\sigma_1=|\gamma|+1$ and $K_1 = \max\{1,(2\lambda)^{-1}\}$.
Combining \eqref{MP:bounda} with \eqref{MP:bound} and \eqref{MP:bound3},
we see that
\begin{eqnarray*}
\sum_{n=0}^\infty |c_n| \sum_{k=0}^n |a_{n,k}||P_k^{(\lambda)}(x;\psi)|\le
K_1\sum_{n=0}^\infty (1+n)^{\sigma_1+\sigma+2\lambda} \left(\frac{|\rho|(\sin\phi+|\sin(\psi-\phi)|)}{\sin\psi}\right)^n<\infty.
\end{eqnarray*}
Reversing the summation and simplifying gives the desired result.
\end{proof}


\section{Definite integrals}
We may apply the orthogonality relation on these continuous hypergeometric orthogonal
polynomials to the generalized generating functions considered above to calculate their
corresponding definite integrals.
The Wilson polynomials satisfy the orthogonality relation given in 
\cite[(9.1.2)]{Koekoeketal}
with weight $w:(0,\infty)\to\R$ 
defined by
\begin{eqnarray} 
\label{W:weight}
w(x):=\left|\frac{\Gamma(a+ix)\Gamma(b+ix)\Gamma(c+ix)\Gamma(d+ix)}{\Gamma(2ix)}\right|^2.
\end{eqnarray}

\begin{cor}\label{I:W1}
Let $\rho\in\C$, $|\rho|<1$, and $a,b,c,d,h$ complex parameters with positive real parts, non-real parameters occurring in conjugate pairs among $a,b,c,d$ and $a,b,c,h$.
Then
\begin{eqnarray} \nonumber
&& \int_{0}^\infty \,
\Hyper{2}{1}{a+ix,c+ix}{a+c}{\rho}\,\Hyper{2}{1}{b-ix,d-ix}{b+d}{\rho} \\
&& \hspace{15mm} \times W_k(x^2;a,b,c,h)
\left|
\frac{\Gamma(a+ix)\Gamma(b+ix)\Gamma(c+ix)\Gamma(h+ix)}{\Gamma(2ix)}\right|^2
dx
\nonumber \\
&& \hspace{15mm} =  2\pi \frac{\Gamma(a+c)\Gamma(k+a+b)\Gamma(k+a+h)\Gamma(k+b+c)\Gamma(k+c+h)\Gamma(k+b+h)}{(b+d)_k\Gamma(2k+a+b+c+h)\{(k+a+b+c+d-1)_k\}^{-1}} \nonumber \\
\label{I:CMVEx1Corr}
&& \hspace{15mm}
\times \,\Hyper{4}{3}{d-h,2k+a+b+c+d-1,k+a+b,k+b+c}{k+a+b+c+d-1,2k+a+b+c+h,k+b+d}{\rho} \rho^k.
\end{eqnarray}
\end{cor}
\begin{proof}
Choose some $k\in\No$. We begin by multiplying both sides of \eqref{W:f1}
by $W_{k}(x^2;a,b,c,h)w(x)$, where $w(x)$ is the weight function as defined in
\eqref{W:weight} with $d$ replaced by $h$. Integrating 
over $x \in (0,\infty)$ causes the infinite sum on the right hand side 
to vanish except for a single term. Simplification produces the
desired result.

In order to justify the interchange of sum and integral we show that the series on the right-hand side of \eqref{W:f1}
converges in the $L^2$-sense with respect to the weight $w(x)$.
This requires 
\begin{eqnarray}\label{I:eq1}
\sum_{k=0}^\infty d_k^2s_k^2 <\infty ,
\end{eqnarray}
where $s_k>0$ is defined by
\begin{eqnarray*}
 s_k^2=\int_0^\infty w(x)\left\{W_k(x)\right\}^2\,dx ,
\end{eqnarray*}
and 
\begin{eqnarray*}
 d_k=\sum_{n=k}^\infty c_na_{n,k} ,
\end{eqnarray*}
where $c_n$ and $a_{n,k}$ are defined as in the proof of Theorem \ref{W:t1}.
From the orthogonality relation corresponding to (\ref{W:weight}) we obtain
\begin{eqnarray*}
 s_k\le K(1+k)^\sigma k!^3 ,
\end{eqnarray*} 
where $K,\sigma$ are positive constants independent of $k$.
Since the estimate of $s_k$ is of the same type as that of $W_k$, by the argument used in the proof of Theorem 1,
we see that, for $|\rho|<1$,
\begin{eqnarray*}
\sum_{n=0}^\infty |c_n|\sum_{k=0}^n |a_{n,k}|s_k <\infty .
\end{eqnarray*} 
This implies
\begin{eqnarray*}
 \sum_{k=0}^\infty |d_k| s_k <\infty
\end{eqnarray*} 
which in turn implies \eqref{I:eq1} and the proof is complete.
\end{proof}

\begin{cor}\label{I:W2}
Let $\rho\in\C$, $|\rho|<1$, and $a,b,c,d,h$ complex parameters with positive real parts, non-real parameters occurring in conjugate pairs among $a,b,c,d$ and $a,b,c,h$.
Then
\begin{eqnarray} \nonumber
&& \int_{0}^\infty
(1-\rho)^{1-a-b-c-d}\,
\Hyper{4}{3}{\frac{1}{2}(a+b+c+d-1),\frac{1}{2}(a+b+c+d),a+ix,a-ix}{a+b,a+c,a+d}{-\frac{4\rho}{(1-\rho)^2}} \\
&& \hspace{15mm} \times W_k(x^2;a,b,c,h)w(x) dx \nonumber \\
&& \hspace{15mm} = \frac{2\pi\Gamma(a+b)\Gamma(a+c)\Gamma(k+a+h)\Gamma(k+b+c)
\Gamma(k+b+h)\Gamma(k+c+h)}{\Gamma(2k+a+b+c+h)(a+d)_k\{(k+a+b+c+d-1)_k(a+b+c+d-1
)_k\}^{-1}} \nonumber \\
&& \hspace{15mm} \times\,
\Hyper{3}{2}{2k+a+b+c+d-1,d-h,k+b+c}{2k+a+b+c+h,a+d+k}{\rho}\rho^k.
\label{I:WilsonInt3}
\end{eqnarray}
\end{cor}
\begin{proof}
The proof is the similar to the proof of Corollary \ref{I:W1}, except apply to
both sides of \eqref{W:f2}.
\end{proof}


The continuous dual Hahn polynomials satisfy the orthogonality relation
\cite[(9.3.2)]{Koekoeketal}.

\begin{cor} \label{I:CDH1}
Let $\rho\in\C$, $|\rho|<1$, and $a,b,c,d,f > 0$, except for possibly a pair of complex conjugates with positive
real parts. Then
\begin{eqnarray} \nonumber
\int_0^\infty (1-\rho)^{-d+ix}\Hyper{2}{1}{a+ix,b+ix}{a+b}{\rho}\left|\frac{
\Gamma(a+ix)\Gamma(d+ix)\Gamma(f+ix)}{\Gamma(2ix)}\right|^2S_k(x^2;a,d,f)dx \\
=
2\pi\frac{\Gamma(k+a+d)\Gamma(k+a+f)\Gamma(k+d+f)}{(a+b)_kk!}
\Hyper{2}{1}{b-f, k+a+d}{k+a+b}{\rho}.
\end{eqnarray}
\end{cor}
\begin{proof}
Multiplying both sides of \eqref{CDH:f1} by $S_k(x^2;a,d,f)\left|\frac{\Gamma(a+ix)
\Gamma(d+ix)\Gamma(f+ix)}{\Gamma(2ix)}\right|^2$ and integrating across $x \in
(0, \infty)$ causes the series to vanish except for a single term. 
Justification for interchanging the sum and the integral is similar to the proof of 
Corollary \ref{I:W1}.  We leave this exercise to the reader.
Simplifying completes the proof.
\end{proof}

\begin{cor} \label{I:CDH2}
Let $\rho\in\C$, and $a,b,c,d > 0$ except for
possibly a pair of complex conjugates with positive real parts among $a,b,c$ and $a,b,d$.
Then
\begin{eqnarray} \nonumber
\int_0^\infty e^\rho\Hyper{2}{2}{a+ix,a-ix}{a+b,a+c}{-\rho}\left|\frac{\Gamma(a+ix)
\Gamma(b+ix)\Gamma(d+ix)}{\Gamma(2ix)}\right|^2S_k(x^2;a,b,d)dx \\
= \frac{\Gamma(a+b)\Gamma(k+a+d)\Gamma(k+b+d)}{(a+c)_k}\Hyper{1}{1}{c-d}{k+a+c}
{\rho}\rho^k.
\end{eqnarray}
\end{cor}
\begin{proof}
The proof is essentially the same as for Corollary \ref{I:CDH1}.
Multiply both sides of \eqref{CDH:f2} by $S_k(x^2;a,b,d)\left|\frac{\Gamma(a+ix)
\Gamma(b+ix)\Gamma(d+ix)}{\Gamma(2ix)}\right|^2$ and integrate over $x \in
(0, \infty)$. With justification for reordering, simplification completes the proof.
\end{proof}

\begin{cor} \label{I:CDH3}
Let $\rho\in\C$ with $|\rho|<1$, $\gamma\in\C$ and $a,b,c,d > 0$ except for
possibly a pair of complex conjugates with positive real parts among $a,b,c$ and $a,b,d$.
Then
\begin{eqnarray} \nonumber
&&\int_0^\infty (1-\rho)^{-\gamma}\Hyper{3}{2}{\gamma,a+ix,a-ix}{a+b,a+c}{\frac
{\rho}{\rho-1}}\left|\frac{\Gamma(a+ix)\Gamma(b+ix)\Gamma(d+ix)}{\Gamma(2ix)}\right|
^2S_k(x^2;a,b,d)dx \\
&&\hspace{3cm} = \frac{(\gamma)_k\Gamma(a+b)\Gamma(k+a+d)\Gamma(k+b+d)}{(a+c)_k}\Hyper{2}{1}
{-d,\gamma+k}{k+a+c}{\rho}\rho^k.
\end{eqnarray}
\end{cor}
\begin{proof}
Multiply both sides of \eqref{CDH:f3} by $S_k(x^2;a,b,d)\left|\frac{\Gamma(a+ix)
\Gamma(b+ix)\Gamma(d+ix)}{\Gamma(2ix)}\right|^2$, integrate over $x \in
(0, \infty)$. Justification for reordering with simplification completes the proof.
\end{proof}


The continuous Hahn polynomials satisfy the orthogonality relation 
\cite[(9.4.2)]{Koekoeketal}.

\begin{cor}\label{I:CH1}
Let $\rho\in\C$, and $a,b,c \in \C$ such that $\Re a>0$, $\Re b>0$, $\Re c>0$ and
$\Im a = \Im b = \Im c$. Then
\begin{eqnarray} \nonumber
&& \int_{-\infty}^\infty \,
\Hyper{1}{1}{a+ix}{\acc}{-i\rho}\,
\Hyper{1}{1}{\bc-ix}{\bcc}{i\rho}
\Gamma(a+ix)\Gamma(c+ix)\Gamma(\ac-ix)\Gamma(\cc-ix)p_k(x;a,c,\ac,\cc)dx \nonumber \\
&& \hspace{5mm} = 2\pi 
\frac{
\Gamma(\acc)\Gamma(a+\cc)^2\Gamma(\ccc)
(k+\acc+\bcc-1)_k(a+\cc)_k(\ccc)_k
(\rho/4)^k
}
{k!(2k+\acc+\ccc-1)\Gamma(\acc+\ccc-1)
(\bcc)_k(\frac{1}{2}(\acc+\ccc-1))_k
}
\nonumber \\
&& \text{\normalsize $\hspace{5mm} \times \,
\Hyper{4}{5}{\frac{1}{2}(\Re a+\Re b+k),\frac{1}{2}(\Re a+\Re b+k+1),\Re a+\Re b+k-\frac12,
\Re b-\Re c}
{\Re a+\Re b+\frac{k-1}{2},\Re a+\Re b+\frac{k}{2},\Re b+\frac{k}{2},
\Re b+\frac{k+1}{2},\Re a+\Re c+k+\frac{1}{2}}{\frac{-\rho^2}{4}}
.$ \normalsize}
\end{eqnarray}
\end{cor}
\begin{proof}
Multiply both sides of \eqref{CH:f1} by $|\Gamma(a+ix)\Gamma(c+ix)|^2
p_k(a,c,\ac,\cc)$. 
Justification for interchanging the sum and the integral is similar to the proof of 
Corollary \ref{I:W1}.  Again, we leave this to the reader.
Integrating over $x \in (-\infty,\infty)$ completes the proof.
\end{proof}

\begin{cor} \label{I:CH2}
Let $\rho\in\C$, $|\rho|<1$, and $a,b,c \in \C$ such that $\Re a>0$, $\Re b>0$, $\Re c>0$ and
$\Im a = \Im b = \Im c$. Then
\begin{eqnarray} \nonumber
&& \int_{-\infty}^\infty (1-\rho)^{1-\acc-\bcc}\,
\Hyper{3}{2}{\Re a+\Re b-\frac12, \Re a+\Re b, a+ix}{\acc, a+\bc}{-\frac{4\rho}{(1-\rho)^2}}\\
&&\hspace{40mm}\times|\Gamma(a+ix)\Gamma(c+ix)|^2p_{k}(x;a,c,\ac,\cc)dx \\
&& \hspace{10mm} 
= \frac{2\pi(-i\rho)^k\Gamma(\acc)\Gamma(\ccc)[\Gamma(a+\cc)]^2 
(\acc+\bcc-1)_k(k+\acc+\bcc-1)_k(a+\cc)_k(\ccc)_k}
{4^k k!(2k+\acc+\ccc-1) \Gamma(\acc+\ccc-1)
(\Re a+\Re c-\frac12)_k(\Re a+\Re b)_k}
\nonumber \\
&& \hspace{10mm} \times \,
\Hyper{2}{1}{\Re a+\Re b+k-\frac{1}{2}, \Re b-\Re c}
{\Re a+\Re c+k+\frac{1}{2}}{\rho^2}
\end{eqnarray}
\end{cor}

\begin{proof}
The proof is the same as Corollary \ref{I:CH1}, except apply to both
sides of \eqref{CH:f2}.
\end{proof}


If $\lambda > 0$ and $\phi \in (0,\pi)$, then the Meixner-Pollaczek polynomials
satisfy the orthogonality relation \cite[(9.7.2)]{Koekoeketal}.

\begin{cor}\label{I:MP1}
Let $\lambda>0$, $\psi,\phi \in (0,\pi)$, and $\rho\in\C$ such that \eqref{MP:cond} holds.
Then
\begin{eqnarray} \nonumber
&&\int_{-\infty}^\infty (1-e^{i\phi}\rho)^{-\lambda+ix}(1-e^{-i\phi}\rho)^{-\lambda-ix}
e^{(2\psi-\pi)x}|\Gamma(\lambda+ix)|^2P_k^{(\lambda)}(x;\psi)dx \\
&& \hspace{15mm} = \left(1-\rho\frac{\sin(\psi-\phi)}{\sin\psi}\right)^{-2\lambda}\frac{2\pi\Gamma(k+2\lambda)\tilde\rho^k}{(2\sin\psi)^{2\lambda}k!},
\end{eqnarray}
where
\begin{eqnarray*}
\tilde\rho= \frac{\rho\sin\phi}{\sin\psi-\rho\sin(\psi-\phi)}.
\end{eqnarray*}
\end{cor}

\begin{proof}
Multiply both sides of \eqref{MP:f1} by $e^{(2\psi-\pi)x}|\Gamma(\lambda
+ix)|^2P_k^{(\lambda)}(x;\psi)$. Integrating over $x\in(-\infty,\infty)$ gives the desired result.
The justification for interchange of the sum and the integral is similar 
to the proof of Corollary \ref{I:W1}.  This is left to the reader.
\end{proof}

\begin{cor} \label{I:MP2}
Let $\lambda>0$, $\rho\in\C$ and $\psi,\phi \in (0,\pi)$. Then
\begin{eqnarray} \nonumber
&& \int_{-\infty}^\infty e^\rho \,
\Hyper{1}{1}{\lambda+ix}{2\lambda}{(e^{-2i\phi}-1)\rho}
e^{(2\psi-\pi)x}|\Gamma(\lambda+ix)|^2P_k^{(\lambda)}(x;\psi)dx \\
&& \hspace{15mm} = \exp
\left(\frac{\rho e^{-i\phi}\sin(\psi-\phi)}{\sin\psi}\right)
\left(\frac{\sin\phi}{\sin\psi}\right)^k\frac{2\pi\Gamma(2\lambda)\rho^k}{e^{ik\phi}(2\sin\psi)^{2\lambda}k!}.
\end{eqnarray}
\end{cor}
\begin{proof}
The proof is the same as Corollary \ref{I:MP1}, except apply to
both sides of \eqref{MP:f2}.
\end{proof}

\begin{cor} \label{I:MP3}
Let $\lambda > 0$, $\gamma, \rho\in\C$ and $\psi, \phi \in (0,\pi)$ be such that \eqref{MP:cond} holds. Then
\begin{eqnarray} \nonumber
&& \int_{-\infty}^\infty (1-\rho)^{-\gamma} \,
\Hyper{2}{1}{\gamma,\lambda+ix}{2\lambda}{\frac{(1-e^{-2i\phi})\rho}{\rho-1}}
e^{(2\psi-\pi)x}|\Gamma(\lambda+ix)|^2P_k^{(\lambda)}(x;\psi)dx \\
&& \hspace{15mm} = \left(1-\frac{\rho\sin(\psi-\phi)}{e^{i\phi}\sin\psi}\right)^{-\gamma-k}\left(\frac{\sin\phi}{\sin\psi}\right)^k \frac{(\gamma)_k\Gamma(2\lambda)2\pi\rho^k}{e^{ik\phi}(2\sin\psi)^{2\lambda}k!}.
\end{eqnarray}
\end{cor}
\begin{proof}
The proof is the same as Corollary \ref{I:MP1}, except apply to
both sides of \eqref{MP:f3}.
\end{proof}



\def\cprime{$'$}

\end{document}